\documentclass[reqno,12pt]{amsart}

\usepackage{enumerate}
\usepackage[margin=1in]{geometry}
\usepackage{ifpdf}
\usepackage{amsmath}
\usepackage{amsfonts}
\usepackage{amssymb}
\usepackage{amsthm}
\usepackage[ocgcolorlinks,hyperfootnotes=false,colorlinks=true,citecolor=blue,linkcolor=blue,urlcolor=blue]{hyperref}
\usepackage{setspace}
\usepackage{amsrefs}
\usepackage{nicefrac}
\usepackage{graphicx}
\usepackage{color}
\usepackage{mathtools}

\newcommand{\ignore}[1]{}

\renewcommand{\Re}{\operatorname{Re}}
\renewcommand{\Im}{\operatorname{Im}}

\newcommand{\abs}[1]{\left\lvert {#1} \right\rvert}
\newcommand{\sabs}[1]{\lvert {#1} \rvert}
\newcommand{\norm}[1]{\left\lVert {#1} \right\rVert}
\newcommand{\snorm}[1]{\lVert {#1} \rVert}

\newcommand{\C}{{\mathbb{C}}}
\newcommand{\R}{{\mathbb{R}}}

\newtheorem{thm}{Theorem}[section]

\newtheorem{prop}[thm]{Proposition}

\newtheorem{cor}[thm]{Corollary}

\newtheorem{lemma}[thm]{Lemma}

\theoremstyle{definition}

\newtheorem{example}[thm]{Example}

\theoremstyle{remark}

\author{Ji\v{r}\'{\i} Lebl}
\thanks{The first author was in part supported by NSF grant DMS-1362337 and
Oklahoma State University's DIG and ASR grants.}
\address{Department of Mathematics, Oklahoma State University,
Stillwater, OK 74078, USA}
\email{lebl@math.okstate.edu}

\author{Alan Noell}
\address{Department of Mathematics, Oklahoma State University,
Stillwater, OK 74078, USA}
\email{noell@math.okstate.edu}

\author{Sivaguru Ravisankar}
\address{School of Mathematics, Tata Institute of Fundamental Research,
Mumbai 400005, India}
\email{sivaguru@math.tifr.res.in}

\date{June 1, 2017}

\ifpdf
\hypersetup{
  pdftitle={Extension of CR functions from boundaries in Cn x R},
  pdfauthor={Jiri Lebl, Alan Noell, and Sivaguru Ravisankar},
  pdfsubject={Several Complex Variables},
  pdfkeywords={Extension of CR functions, Levi-flat Plateau problem, Hartogs-Bochner, CR singularity},
}
\fi

\title%
{Extension of CR functions from boundaries in $\C^n \times \R$}

\begin{document}


\begin{abstract}
Let $\Omega \subset \C^n \times \R$ be a bounded domain with smooth
boundary such that $\partial \Omega$ has only nondegenerate
elliptic CR
singularities,
and let $f \colon \partial \Omega \to \C$ be a smooth function that is CR at CR
points of $\partial \Omega$ (when $n=1$
we require separate holomorphic extensions for each real parameter).  Then
$f$ extends to a smooth CR function on $\overline{\Omega}$, that is, an
analogue of Hartogs-Bochner holds.
In addition, if $f$ and $\partial \Omega$ are real-analytic, then $f$ is the restriction of a function
that is holomorphic on
a neighborhood of $\overline{\Omega}$ in $\C^{n+1}$.
An immediate application is a (possibly singular) solution of the Levi-flat
Plateau problem for codimension 2 submanifolds that are CR images of
$\partial \Omega$ as above.
The extension also holds locally near nondegenerate, holomorphically
flat, elliptic CR singularities.
\end{abstract}

\maketitle



\section{Introduction} \label{section:intro}

Let $\Omega \subset \C^n \times \R$ be a bounded domain with smooth
boundary, with coordinates
$(z,s) \in \C^n \times \R$.
We ask:
\emph{When does a smooth (i.e., $C^\infty$) function $f \colon \partial \Omega \to \C$ extend
smoothly to $\overline{\Omega}$ such that the extension is holomorphic in the first $n$
variables, that is, holomorphic in $z$?}
As $\Omega$ is Levi-flat, we want the extension to be a CR
function on $\Omega$ and smooth on $\overline{\Omega}$.
Clearly, for $n > 1$, for a fixed $s$, the function $f$ must be a CR function to begin
with.  I.e., it must satisfy the tangential Cauchy-Riemann equations outside the CR
singularities of $\partial \Omega$.
When $n=1$, $f$ must extend holomorphically, separately for each fixed $s$.
Our model case is the sphere:
\begin{equation}
\Omega
= \{ (z,s) \in \C^n \times \R : \snorm{z}^2+\abs{s}^2 < 1 \} . 
\end{equation}
We prove that the conditions given above (CR for $n > 1$ and ``extendible
along leaves'' for $n=1$) are sufficient under
very natural conditions on the CR
singularities (e.g., the north and south pole of the sphere),
that is, the singularities are nondegenerate and elliptic.

A real submanifold $M \subset \C^{n+1}$ of
real codimension 2
generically has isolated CR singularities.
CR singular submanifolds of codimension 2 were first studied in
$\C^2$ by
E.~Bishop~\cite{Bishop65}, who found that 
such nondegenerate submanifolds $M$ are locally of the form
\begin{equation}
w = z\bar{z} + \lambda (z^2+\bar{z}^2) + O(3) ,
\end{equation}
where $0 \leq \lambda \leq \infty$ ($\lambda = \infty$ is interpreted
appropriately).  The constant $\lambda$ is called the Bishop invariant.
If $\lambda < \frac{1}{2}$ then $M$ is said to be elliptic; the expression
$z\bar{z} + \lambda (z^2+\bar{z}^2) = \text{constant}$ gives ellipses in the
$z$ directions.  The case $\lambda = \frac{1}{2}$ is called parabolic,
and $\lambda > \frac{1}{2}$ is called hyperbolic.
Bishop's motivation was to understand the hull of holomorphy of
$M$ via attaching analytic discs to $M$.  In the elliptic case, fixing
$w$ at a fixed constant gives a family of analytic discs shrinking down to
zero as $w \to 0$.
The work of Bishop in $\C^2$,
especially in the elliptic case, has been refined by
Moser-Webster~\cite{MoserWebster83},
Moser~\cite{Moser85},
Kenig-Webster
\cites{KenigWebster:82, KenigWebster:84}, Gong~\cite{Gong94:duke},
Huang-Krantz~\cite{HuangKrantz95}, 
Huang~\cite{Huang:jams}, Huang-Yin~\cite{HuangYin09}, and others.

We are interested in the generalization of the elliptic case 
to higher dimensions.  For prior work on the local case,
mainly on the normal form of CR
singularities, see Huang-Yin~\cites{HuangYin09:codim2,HuangYin:flattening},
Gong-Lebl~\cite{GongLebl}, Burcea~\cites{Burcea,Burcea2}.
In general, it is not possible to
flatten a CR singular submanifold $M$, to realize it as a submanifold
of the Levi-flat hypersurface $\{ \Im w = 0 \}$.
In dimension 3 and higher, existence of such a Levi-flat hypersurface requires
nongeneric conditions on $M$, see \cites{DTZ,HuangYin:flattening}.
In the elliptic case we think of one side of
this hypersurface as a submanifold $H$ with boundary $M$, where $H$ is 
the holomorphic hull of $M$.
Dolbeault-Tomassini-Zaitsev~\cites{DTZ,DTZ2}
considered a Levi-flat analogue of the Plateau problem.  They proved that
a compact CR singular
submanifold satisfying certain conditions (in particular elliptic
CR singularities and a degenerate Levi-form at CR points) is the boundary of
a (singular) Levi-flat hypersurface.

Harris~\cite{Harris} studied the
extension of real-analytic CR functions near CR singularities.
For $M \subset \C^m$, Harris provides a
criterion for extension if
$\dim_{\R} M \geq m$, but the condition can be difficult to verify.
The local extension in the real-analytic
case is possible to obtain using the work of Harris, but we provide
a different proof.
Recently in \cite{LMSSZ} nonextensibility was studied for
a (nonelliptic) CR singular manifold that is a CR image of a CR manifold.

We study nondegenerate, holomorphically flat, and elliptic CR singularities.
See section \ref{section:nondegflatelliptic} for the precise definitions.
These conditions are natural for the extension problem.  Some sort of
nondegeneracy is required, flatness is the setting of the problem, 
and ellipticity
gives bounded domains
shrinking in size as we approach the singularity.
The conditions are an analogue of strict pseudoconvexity for flat CR singular submanifolds.
Nondegenerate elliptic holomorphically flat CR singularities can be put into the following form
locally:
\begin{equation} \label{eq:typeofsing}
M: w = \sum_{j=1}^{n} \bigl( \sabs{z_j}^2 + \lambda_j (z_j^2 + \bar{z}_j^2)
\bigr) +
E(z,\bar{z}) ,
\end{equation}
where $0 \leq \lambda_j < \frac{1}{2}$, and
$E$ is $O(3)$, smooth, and real-valued.
The Levi-flat hypersurface $H$
whose boundary is $M$ is given by
\begin{equation} \label{eq:typeofsingh}
H:  \begin{cases}
\Re w \geq \sum_{j=1}^{n} \bigl( \sabs{z_j}^2 + \lambda_j (z_j^2 + \bar{z}_j^2)\bigr) + E(z,\bar{z}) , \\
\Im w = 0.
\end{cases}
\end{equation}

Consider $H$ and $M$ in $U = \{ (z,w) \in \C^n \times \C : 
\snorm{z} < \delta_z, \sabs{w} < \delta_w \}$.  We say 
$\delta_z,\delta_w >0$ are \emph{small enough} if for all $\sabs{w_0} < \delta_w$
the set $\{ w = w_0 \} \cap H$ is compact or empty, and such that
$M$ has a CR singularity only at the origin.  For an elliptic
singularity we can always pick small enough $\delta_z,\delta_w > 0$.

\begin{thm} \label{thm:mainlocal}
Suppose $H$ and $M$ are closed submanifolds of $U = \{ (z,w) \in \C^n \times \C : 
\snorm{z} < \delta_z, \sabs{w} < \delta_w \}$ given by
\eqref{eq:typeofsing} and \eqref{eq:typeofsingh}, $E$ is real-valued, 
$0 \leq \lambda_j < \frac{1}{2}$ for all $j$ (i.e.\ $M$ is nondegenerate,
holomorphically flat, and elliptic) and $\delta_z,\delta_w > 0$ are small
enough.

Suppose $f \colon M \to \C$ is smooth and either
\begin{enumerate}[(i)]
\item $n > 1$ and $f$ is a CR function on $M_{CR}$ (the CR points of $M$), or
\item $n = 1$ and for every $0< c < \delta_w$, there exists a 
continuous function on $H \cap \{ w = c \}$, holomorphic on $(H \setminus M)
\cap \{ w = c \}$, extending $f|_{M \cap \{ w = c \}}$.
\end{enumerate}

Then there exists a function $F \in C^\infty(H)$ such that
$F$ is CR on $H \setminus M$ and $F|_M = f$.
Furthermore, $F$ has a formal power series at 0 in $z$ and $w$.
If $M$ and $f$ are real-analytic, then 
$F$ is a restriction of a holomorphic function defined in a
neighborhood of $H$ in $\C^{n+1}$ (in particular $F$ is real-analytic).
\end{thm}

The case of $n=1$ is somewhat different:  $M$ is totally-real
at CR points, and any function is a CR function.  Most
smooth functions on $M$ do not extend to $H$ as CR functions at all; therefore the further condition 
that $f$ extend along the ``leaves'' of $H$ is natural.  The
condition is equivalent to the vanishing of certain Fourier
coefficients, that is
for each $s>0$ and $\ell \geq 0$,
\begin{equation}
\smashoperator[r]{\int\limits_{M\cap\{w = s\}}}
f\left(\zeta,\bar{\zeta}\right)\zeta^\ell \, d\zeta = 0 .
\end{equation}
See the proof of Lemma~\ref{lem:OrderK}.
We note that in the local case, we always parametrize $M$ by $z$,
and we think of $f$ as a function of $z$, so we write $f(z,\bar{z})$.

We also prove a global extension 
with the aid of the Bochner-Martinelli kernel.

\begin{thm} \label{thm:mainglobal}
Suppose
$\Omega \subset \C^n \times \R$ is a bounded domain with smooth boundary.
Let $(z,s) \in \C^n \times \R$ be the coordinates.
Suppose all CR singularities of $\partial \Omega$ are nondegenerate
and elliptic.
Suppose $f \colon \partial \Omega \to \C$ is smooth and
either
\begin{enumerate}[(i)]
\item $n > 1$ and $f$ is a CR function on ${(\partial \Omega)}_{CR}$, or
\item $n = 1$ and for every $c \in \R$ where
$\Omega \cap \{ s = c \}$ is nonempty, there exists a
continuous function on $\overline{\Omega} \cap \{ s = c \}$, holomorphic on
$\Omega \cap \{ s = c \}$, extending $f|_{\partial \Omega \cap \{ s = c \}}$.
\end{enumerate}

Then there exists a function $F \in C^\infty(\overline{\Omega})$ such that
$F$ is CR on $\Omega$ and $F|_{\partial \Omega} = f$.
Furthermore, if $\partial \Omega$ and $f$ are real-analytic, then
$F$ is a restriction of a holomorphic function defined in a
neighborhood of $\overline{\Omega}$ in $\C^{n+1}$
(in particular $F$ is real-analytic).
\end{thm}

We think of this theorem as an analogue of Hartogs-Bochner
in $\C^{n} \times \R$
with smooth data.
It should be noted that the hypotheses on the CR singularities imply
$\partial \Omega$ is homeomorphic to a sphere (see
Proposition~\ref{prop:topology}) and so 
$\partial \Omega$ is connected.
Using the global version of the extension, we immediately obtain the following
(singular) solution to the Levi-flat Plateau problem for CR images of
submanifolds of $\C^n \times \R$.

\begin{cor}
Suppose $\Omega \subset \C^{n} \times \R$, $n > 1$, is a bounded domain with smooth
boundary, and $M = f(\partial \Omega) \subset \C^{n+1}$ is the image
of a smooth map $f$ that is CR on ${(\partial \Omega)}_{CR}$.  Suppose all
CR singularities of $\partial \Omega$ are nondegenerate 
and elliptic.
Then there exists a smooth map $F \colon \overline{\Omega} \to \C^{n+1}$ such that
$F$ is CR on $\Omega$ and $F|_{\partial \Omega} = f$ (in particular $F(\partial \Omega) =
M$).  Wherever
$F(\overline{\Omega})$ is a smooth real-hypersurface, it is Levi-flat.

Furthermore, if $f$ and $\partial \Omega$ are real-analytic,
then $F$ is the restriction of a holomorphic function defined in a
neighborhood of $\overline{\Omega}$ in $\C^{n+1}$.
\end{cor}

In particular, we have $F(\overline{\Omega})
\subset \widehat{M}$, the holomorphic (polynomial) hull of $M$.

Theorem~\ref{thm:mainlocal} has another interesting consequence.  Certain
CR singular submanifolds arise as images of CR submanifolds under CR maps
that are diffeomorphisms onto their image.  Our theorem combined with results
of~\cite{LMSSZ} shows that nondegenerate holomorphically flat elliptic
real-analytic submanifolds are never such images when $n > 1$.
In particular, suppose $f \colon M \to \C^{n+1}$ is a
real-analytic CR map from a real-analytic CR manifold $M$
that is a diffeomorphism onto its image $f(M)$,
with $f(M) \subset \C^{n+1}$ CR singular at $p \in f(M)$.  In~\cite{LMSSZ}
it was proved 
there exists a real-analytic function $u$ on $f(M)$, vanishing
on all CR vectors tangent to $f(M)$, yet $u$ is not a restriction
of a holomorphic function.  If $n > 1$, then Theorem~\ref{thm:mainlocal} says
$f(M)$ cannot be nondegenerate, holomorphically flat, and elliptic at $p$.
If $n=1$, then every Bishop surface is locally a diffeomorphic image of $\R^2$,
which does not contradict Theorem~\ref{thm:mainlocal}, as not every
real-analytic function extends in this case.

The organization of this article is as follows.
In section~\ref{section:nondegflatelliptic} we introduce the basic
terminology.  
In section~\ref{section:topology} we recall from \cite{DTZ} that the
relevant domains are topologically spheres and discuss the topology of the
leaves.  In section~\ref{section:awayfromsing} we prove the extension
outside the CR singular points.  We then prove the extension for the
model case for polynomials in section~\ref{section:poly}.  Using the
polynomial extension we show the smooth case of Theorem~\ref{thm:mainlocal}
in section~\ref{section:smooth} and combining with the previous results we prove the
smooth case of Theorem~\ref{thm:mainglobal}.  In section~\ref{section:realanal}
we prove the real-analytic assertions of 
Theorems \ref{thm:mainlocal} and \ref{thm:mainglobal}.  Finally we briefly
discuss what happens in the case of an elliptic-like degenerate CR singularity in
section~\ref{section:degenerate}.


\section{Nondegenerate holomorphically flat elliptic CR singularities}
\label{section:nondegflatelliptic}

Let us review some well-known results and fix the meaning of the
definitions.
We work in $\C^{n+1}$ and we assume $n \geq 1$.
Let $M \subset \C^{n+1}$ be a real submanifold,
and let $T_p^cM = J(T_pM) \cap T_pM$, where
$J$ is the complex structure on $\C^{n+1}$ (that is, multiplication by $i$).
We write $\C \otimes T_p^c M = T_p^{(1,0)} M \oplus T_p^{(0,1)} M$,
where $T_p^{(1,0)} M$ are the holomorphic and  $T_p^{(0,1)} M$ the
antiholomorphic vectors tangent to $M$ at $p$.
We say $M$ is \emph{CR} at $q \in M$ if the dimension of $T_p^cM$ (or $T_p^{(1,0)}
M$) is constant as $p$ varies near $q$. 
Denote by $M_{CR}$ the set of CR points of $M$.  We say $M$ is \emph{CR
singular} if
there exists a non-CR point.

A CR singular submanifold $M \subset \C^{n+1}$ of real
codimension 2 can locally near the singularity be given (after a rotation and translation) as
\begin{equation}
w =
Q(z,\bar{z})
+ E(z,\bar{z})
=
\sum_{j,k=1}^{n} \left( a_{jk} z_j \bar{z}_k
+
b_{jk} z_j z_k
+
c_{jk} \bar{z}_j \bar{z}_k
\right)
+
E(z,\bar{z}) ,
\end{equation}
where $E$ is $O(3)$, and $(z,w) \in \C^n \times \C$.
We make
$c_{jk} = \bar{b}_{jk}$ by absorbing some terms into $w$ via a biholomorphic
change of variables.  We choose $[b_{jk}]$ to be a symmetric matrix.

First by \emph{nondegenerate} we mean that $[a_{jk}]$ is a nonsingular
matrix.  Although we do not use this fact, we remark that the
matrix $[a_{jk}]$ is related to the Levi-form.  That is, the Levi-form of
$M$ at CR points is a perturbation of the $(n+1) \times (n+1)$ matrix
$[a_{jk}] \oplus [0]$ restricted to an $(n-1)$ dimensional subspace.

By \emph{holomorphically flat} we mean that $M$ is contained in a real-analytic Levi-flat
hypersurface, or equivalently there exists a holomorphic function near
$0$ with nonvanishing derivative that is real-valued on $M$.
We arrange this function to be $w$ and hence the hypersurface to be
given by $\Im w = 0$ (see \cite{DTZ}). 
Therefore when $M$ is holomorphically flat we assume 
$[a_{jk}]$ is a Hermitian matrix and $E(z,\bar{z})$ is
real-valued.

Finally a holomorphically flat nondegenerate $M$ is \emph{elliptic} if
the sets $\{ z : Q(z,\bar{z}) = \text{constant} \}$
are either empty or (real) ellipsoids.  Via basic linear algebra,
being elliptic implies $[a_{jk}]$ must be positive definite and
the eigenvalues of $[b_{jk}]$ are small compared to $[a_{jk}]$.
We will make their size precise below.
Equivalently for small enough real constants $c$,
the sets $M \cap \{ w = c \}$ are
either empty or compact smooth manifolds whose diameter goes to 0 as $c$ goes to 0.
When we say that a boundary of a domain $\Omega \subset \C^n \times \R$ has an
elliptic CR singularity we mean that $\Omega$ corresponds to 
$s > Q(z,\bar{z})+E(z,\bar{z})$.

It is a classical result in linear algebra that a positive definite
matrix $A$ and a symmetric matrix $B$ can be diagonalized using
$T^*AT$ and $T^tBT$ for a nonsingular matrix $T$ where $T^* =
\overline{T^t}$ is the conjugate transpose.
See e.g.\ \cite{HornJohnson}*{Theorem 7.6.6}.
Hence 
every nondegenerate holomorphically flat elliptic CR singularity can be put into the form
\begin{equation} \label{eq:typeofsing2}
w = \sum_{j=1}^{n} \bigl(\sabs{z_j}^2 + \lambda_j (z_j^2 + \bar{z}_j^2)\bigr) +
E(z,\bar{z}) .
\end{equation}
Clearly we can make $\lambda_j \geq 0$.  Furthermore the
eigenvalues must satisfy $\lambda_j < \frac{1}{2}$ if $E$ is to be elliptic.
The numbers $\lambda_j$ are called the \emph{Bishop invariants} of $M$.

Suppose $M$ is given by \eqref{eq:typeofsing2}.  We write $M^{quad}$
for the quadratic model.  That is, 
\begin{equation} \label{eq:Mquad}
M^{quad}:  w = \sum_{j=1}^{n} \bigl(\sabs{z_j}^2 + \lambda_j (z_j^2 +
\bar{z}_j^2)\bigr) .
\end{equation}

Let us give some examples to show the conditions above are natural
for the extension problem.

\begin{example}
The sphere from the introduction can locally be given by
\begin{equation}
w = \snorm{z}^2 . 
\end{equation}
Our results say every smooth function $f$ on $M$ that is CR ($n > 1$)
on $M_{CR}$ extends to a CR function on
\begin{equation}
\Re w \geq \snorm{z}^2 , \qquad \Im w = 0,
\end{equation}
smooth up to the boundary.  However there does not necessarily exist
a
holomorphic function of $(z,w)$ whose restriction is $f$.  For example 
a smooth non-real-analytic function depending only on $\Re w$ is
clearly CR, and does not extend to a holomorphic function.

Furthermore, take a holomorphic function $F$ on $\{ \Re w > \snorm{z}^2 \}$
smooth on $\{ \Re w \geq \snorm{z}^2 \}$ that does not extend
holomorphically past any point.
Letting $f = F|_M$ we obtain a smooth CR function that extends to $H$,
but does not extend holomorphically
through the CR singularity.  In this sense the nondegenerate elliptic
singularity is analogous to a strongly pseudoconvex point.
\end{example}

\begin{example}\label{eg:NonDegen}
Nondegeneracy is needed for smoothness up to the boundary.  For example,
let $M$ and $H$ be defined by
\begin{equation}
M : 
w = \snorm{z}^4 ,
\qquad
H : 
\Re w \geq \snorm{z}^4 , ~ \Im w = 0 .
\end{equation}
The function $f \colon M \to \C$ given by $f = \sqrt{\Re w}$ is clearly CR
on $M_{CR}$.  It clearly
extends along every leaf $\{ w = \text{constant} \}$.  The function $f$ is smooth on $M$
since
\begin{equation}
\sqrt{\Re w}= \sqrt{\snorm{z}^4} =
\snorm{z}^2 \quad \text{on } M.
\end{equation}
It is smooth in the
interior of $H$, and up to the boundary on $H \setminus \{ 0 \}$, but it
is not smooth at the origin; the normal
derivative blows up.  It is however continuous on $H$.
\end{example}


\section{The topology}
\label{section:topology}

Before we prove the extension let us briefly discuss the topology of the
given domains.  The boundary of a smoothly bounded domain
$\Omega \subset \C^n \times \R$ has a CR singularity precisely when the
boundary is tangent to one of the leaves, that is, $\C^n \times \{ c \}$ for a
constant $c \in \R$.

The equations $Q(z,\bar{z}) = c$ define ellipsoids, which are
homeomorphic to spheres.
Therefore sufficiently near the origin the equations
$Q(z,\bar{z})+E(z,\bar{z}) =c$ define
bounded domains also homeomorphic to spheres.
The proof of the following proposition follows easily.

\begin{prop}
Suppose $M$ is given by
\begin{equation}
w = \sum_{j=1}^{n} \bigl( \sabs{z_j}^2 + \lambda_j (z_j^2 + \bar{z}_j^2)
\bigr) +
E(z,\bar{z}) ,
\end{equation}
where $0 \leq \lambda_j < \frac{1}{2}$, and
$E$ is $O(3)$, smooth, and real-valued.
Then there exist a $\delta_z > 0$ and a $\delta_w > 0$ such that
for any $0 < c < \delta_w$ the set
\begin{equation}
\{ (z,c)\in M : \snorm{z} < \delta_z \},
\end{equation}
is either empty or a connected compact real hypersurface homeomorphic to a
sphere, in particular bounding a
relatively compact domain with connected boundary.
\end{prop}

Next we prove that bounded domains in $\C^n \times \R$ with
elliptic CR singular points are topologically spheres.
This proposition has been essentially proved in \cite{DTZ}.
Let $(z,s) \in \C^n \times \R$ be the coordinates.

\begin{prop} \label{prop:topology}
Suppose
$\Omega \subset \C^n \times \R$ is a bounded domain with smooth boundary.
Suppose all CR singularities of $\partial \Omega$ are nondegenerate
and elliptic.  
Then
\begin{enumerate}[(i)]
\item $\partial \Omega$ is homeomorphic to a sphere,
\item $\partial \Omega$ has exactly two CR singular points,
\item for each $c \in \R$ the set
$\Omega \cap ( \C^n \times \{ c \} )$ is a smoothly bounded domain with
connected boundary.
\end{enumerate}
\end{prop}

The proof is easy in our case and we give a version here for the reader's convenience.

\begin{proof}
As $\Omega$ is bounded, there must be at least two CR singularities,
one for the infimum and one for the supremum of the $s$ that hit $\Omega$.
Let $(s_1,s_2)$ be the largest open interval such that for all
$c \in (s_1,s_2)$
the set $\bigl( \C^n \times \{ c \} \bigr) \cap \Omega \not= \emptyset$.

The CR singular points are exactly those points where a plane $\C^n \times \{
c\}$ is tangent to $\partial \Omega$.  Therefore
for an interval $I \subset (s_1,s_2)$ such that $\partial \Omega \cap \{ s
\in I \}$ contains no CR singular points, it is easy to see that
for all $c \in I$, the topology $\partial \Omega \cap \{ s = c \}$ is
identical.  All the CR singularities are ``caps''; they either begin or
end a component of $\partial \Omega \cap \{ s = c \}$.  Therefore
the entire subset $(s_1,s_2)$ must be such an interval $I$.  So CR
singularities are only allowed at $s_1$ and $s_2$ and as $\Omega$ is
bounded, we are only allowed one component so we are only allowed one CR
singularity at each end.  The existence of the homeomorphism follows.
\end{proof}


\section{Away from CR singularities}
\label{section:awayfromsing}

Let us start with all the points of the hypersurface except the CR
singularity of the boundary.
First we look at the transversal derivative separately in the $n=1$ case.

It will be useful to use the following notation.
For sets $X \subset \C^n \times \R$ and 
$I \subset \R$, define
\begin{equation}
(X)_I := \{ (z,s) \in X : s \in I \}.
\end{equation}

\begin{lemma} \label{lem:SmExtnHar}
Let $\Omega \subset \C \times \R$ be a bounded domain with smooth boundary
and
let $I\subset \mathbb{R}$ be an open interval such that
$(\partial \Omega)_I$ contains no CR singularities of
$\partial\Omega$, and such that $(\partial \Omega)_{\{c\}}$ is
connected for all $c \in I$.  Let $(z=x+iy,s)$ be the coordinates.
Suppose $F \colon (\overline{\Omega})_I \to \C$
is a continuous function such that
$F \in C^\infty\bigl((\Omega)_I\bigr)$ and
$F \in C^\infty\bigl((\partial \Omega)_I\bigr)$,
and for all $c \in I$,
$F|_{(\Omega)_{\{c\}}}$ is harmonic.
Finally suppose
for all $j,k \geq 0$,
$\frac{\partial^{j+k}F}{\partial x^j \partial y^k}
\in C\bigl((\overline{\Omega})_I\bigr) \cap
C^\infty\bigl((\partial \Omega)_I\bigr)$.

Then the derivative $F_s \in C\bigl((\overline{\Omega})_I\bigr) \cap
C^\infty\bigl((\partial \Omega)_I\bigr)$.
\end{lemma}

\begin{proof}
It is enough to prove the lemma for a small subinterval $J$ such that
$\overline{J} \subset I$.  We therefore assume that $F$ and its
$x$ and $y$ derivatives are continuous on $(\overline{\Omega})_{\overline{I}}$
and smooth on $(\partial \Omega)_{\overline{I}}$, both of which are compact
sets.

Let $G(x,y,s)$ be the real or imaginary part of $F$.
Fix some $s_0 \in I$.
Denote by $V \subset \R^2$ the domain $(\Omega)_{\{s_0\}}$ in $(x,y)$ coordinates.
Define $\lambda(x,y,\epsilon)$ and
$\omega(x,y,\epsilon)$
with $\lambda(x,y,0) = \omega(x,y,0) = 0$,
so that for small $\epsilon$ the map
\begin{equation}
\Phi(x,y,\epsilon) := \bigl(x+\lambda(x,y,\epsilon),y+\omega(x,y,\epsilon),s_0+\epsilon \bigr)
\end{equation}
diffeomorphically maps $\overline{V} \times [-\delta,\delta]$
to $(\overline{\Omega})_{[s_0-\delta,s_0+\delta]}$ for some small $\delta > 0$.
Note that $(\Omega)_{\{s_0\}}$ is taken to $(\Omega)_{\{s_0+\epsilon\}}$,
and $\Phi$ is the identity if $\epsilon = 0$.  The functions $\lambda$
and $\omega$ are smooth up to the boundary and are $O(\epsilon)$.

Let
\begin{equation}
G^\epsilon (x,y) := G\bigl(x+\lambda(x,y,\epsilon),
\, y+\omega(x,y,\epsilon),\, s_0 + \epsilon \bigr) .
\end{equation}
Notice $G^0(x,y) = G(x,y,s_0)$ is harmonic on $V$.
We compute
\begin{multline}
\nabla^2 (G^\epsilon - G^0) = 
(G_{xx} \circ \Phi)(2 \lambda_x + \lambda_x^2 + \lambda_y^2)
+ (G_{yy} \circ \Phi)(2 \omega_y + \omega_y^2 + \omega_x^2)
\\
+ (G_{xy} \circ \Phi)\bigl(2(1+\lambda_x)\omega_x + 2(1+\omega_y)\lambda_y\bigr)
\\
+ (G_x \circ \Phi)(\lambda_{xx}+\lambda_{yy})
+ (G_y \circ \Phi)(\omega_{xx}+\omega_{yy}) .
\end{multline}
All the $x$ and $y$ derivatives of $G$ are
continuous in $(x,y) \in \overline{V}$, and $\epsilon$,
and these are multiplied
by smooth functions of 
$(x,y) \in \overline{V}$ and $\epsilon$, which are $O(\epsilon)$.
That is, each term in the sum is of the form $A(x,y,\epsilon)
B(x,y,\epsilon)$ where $A$ is continuous and $B$ smooth and vanishing to
first order in $\epsilon$.
In particular,
\begin{equation}
\abs{\nabla^2 (G^\epsilon - G^0)} \leq C \epsilon .
\end{equation}
By putting sub and superharmonic functions above and below,
for example if $V$ is contained in a disc of radius $R$, then adding
$\pm \frac{1}{4} R^2 C \epsilon \bigl(1-{(x/R)}^2-{(y/R)}^2\bigr)$,
we obtain for any $(x_0,y_0) \in V$
\begin{equation}
\abs{G^{\epsilon}(x_0,y_0)-G^0(x_0,y_0)} \leq 
\sup_{(x,y)\in \partial V} \abs{G^{\epsilon}(x,y)-G^0(x,y)} + C'\epsilon 
\end{equation}
for some other constant $C'$.  Thus 
\begin{equation}
\sup_{(x,y)\in V} \abs{\frac{G^{\epsilon}-G^0}{\epsilon}} \leq 
\sup_{(x,y)\in \partial V} \abs{\frac{G^{\epsilon}-G^0}{\epsilon}} + C'
\leq
\sup_{(x,y)\in \partial V, t \in [-\epsilon,\epsilon]}
\abs{\frac{\partial}{\partial t} \Bigl[ (G \circ \Phi) (x,y,t) \Bigr] } + C' .
\end{equation}
As $G \circ \Phi$ is smooth on the compact set
$\partial V \times [-\epsilon,\epsilon]$,
we obtain $\frac{\partial}{\partial\epsilon}G^\epsilon$ and
therefore $G_s$ is bounded on $V$.  By applying the same argument
to $G_x$ and $G_y$ we obtain $G_{sx}$ and $G_{sy}$ are also bounded on
$V$.

Next let us consider $G_{ss}$.
The expression
$\nabla^2 (G^\epsilon+G^{-\epsilon} - 2G^0)$
has terms of the form
\begin{multline}
A(x,y,\epsilon) B(x,y,\epsilon)
+
A(x,y,-\epsilon) B(x,y,-\epsilon)
=
\\
\Bigl(A(x,y,\epsilon) - A(x,y,-\epsilon) \Bigr) B(x,y,\epsilon)
+
A(x,y,-\epsilon) \Bigl( B(x,y,\epsilon) + B(x,y,-\epsilon) \Bigr) .
\end{multline}
Let us work in $V$, where all derivatives make sense.
As $G_{xs}$, $G_{ys}$, $G_{xxs}$, $G_{xys}$, $G_{yys}$ are all bounded, the derivative $A_\epsilon$ is also bounded.
By expanding $A(x,y,\epsilon) - A(x,y,-\epsilon)$
up to first order (mean value theorem) we find
$\sabs{A(x,y,\epsilon) - A(x,y,-\epsilon)} \lesssim \epsilon$.
Since $B$ is smooth up to the boundary then $\sabs{B(x,y,\epsilon)} \lesssim \epsilon$.
So the product is $\lesssim \epsilon^2$.  As $A$ is continuous up to the
boundary and so
bounded, and $B$ is smooth up to the boundary, the right hand term is also
$\lesssim \epsilon^2$.  That is, there exists a constant $C$ such that for
all small $\epsilon$ and $(x,y) \in V$:
\begin{equation}
\abs{\nabla^2 (G^\epsilon+G^{-\epsilon} - 2G^0)} \leq C \epsilon^2 .
\end{equation}
As above
\begin{equation}
\sup_{(x,y) \in V} \abs{G^{\epsilon}+G^{-\epsilon}-2G^0} \leq 
\sup_{(x,y) \in \partial V} \abs{G^{\epsilon}+G^{-\epsilon}-2G^0} + C'\epsilon^2 .
\end{equation}
We divide by $\epsilon^2$ and obtain
that the second derivative in $\epsilon$ of $G^{\epsilon}$ is bounded.  As
all the other second and all first derivatives are bounded, then
$G_{ss}$ is bounded on $V$.  So $G_s$ extends continuously to
$\overline{V} \times [s_0-\delta,s_0+\delta]$, and therefore $F_s$ extends
continuously to $(\overline{\Omega})_I$.

To show that $F_s$ is smooth on the boundary,
we use the mean value theorem along the sides of a `box' obtained as follows;
start with $(z,s)\in(\partial \Omega)_I$, push it inside along a leaf by
$\approx\epsilon$, then move in the $s$ direction by $\approx\epsilon$, and
finally project back to $(\partial \Omega)_I$ along a leaf.

Suppose that $\rho$ is a defining function for $\partial \Omega$ near $(z,s)$.
We know 
$\rho_{z}(z,\bar{z},s)\neq 0$ and by a rotation in the $z$ plane we may suppose $\rho_{z}(z,\bar{z},s) > 0$.
Hence $(z-\epsilon,s)\in(\Omega)_I$ for small $\epsilon>0$.
Since $\rho\bigl(z-\epsilon,\bar{z}-\epsilon, s) \approx -\epsilon$ and $\rho_s$ is locally
bounded, there is a smooth real-valued function $\eta(\epsilon)$, such that
$\eta(\epsilon) =O(\epsilon)$ and
$\bigl(z-\epsilon,s+\eta(\epsilon)\bigr)\in(\Omega)_I$.
In addition, we choose $\eta$ so that $\eta(\epsilon)\to 0$ as $\epsilon\to 0$ and $\eta'(0)\neq 0$.
We also find a smooth real-valued function $\lambda(\epsilon)$ such that
$\bigl(z+\lambda(\epsilon), s+\eta(\epsilon)\bigr)\in(\partial\Omega)_I$ and $\lambda(\epsilon)\to 0$ as $\epsilon\to 0$.

For simplicity let us use $F(x,y,s)$ instead of $F(z,\bar{z},s)$.
Hence,
\begin{equation}
\begin{split}
F\bigl(x-\epsilon,y,s+\eta(\epsilon)\bigr) - F(x-\epsilon,y,s)
& =
 F\bigl(x-\epsilon,y,s+\eta(\epsilon)\bigr)-F\bigl(x+\lambda(\epsilon),y,s+\eta(\epsilon)\bigr)\\
 & \phantom{=} + F\bigl(x+\lambda(\epsilon),y,s+\eta(\epsilon)\bigr)- F(x,y,s)
\\
 & \phantom{=} + F(x,y,s)- F(x-\epsilon,y,s).
 \end{split}
\end{equation}

Now, we use the mean value theorem to get 
\begin{multline}
F_s(x-\epsilon,y,s+\eta_1)\eta(\epsilon) =  F_{x}\bigl(x-\epsilon_1,y,s+\eta(\epsilon)\bigr)\bigl(-\epsilon-\lambda(\epsilon)\bigr) \\
 +YF(z^{\epsilon},s^{\epsilon}) \gamma(\epsilon) +F_{x}(x-\epsilon_2,y,s)(\epsilon).
\end{multline}
where $\gamma$ is a smooth function with $\gamma(\epsilon)=O(\epsilon)$, $Y$
is locally a smooth vector field on $(\partial\Omega)_I$ near $(z,s)$, $(\partial\Omega)_I\ni (z^{\epsilon},s^{\epsilon})\to (z,s)$ as
$\epsilon\to 0$, $0\leq \eta_1\leq \eta(\epsilon)$, $0 \leq \sabs{\epsilon_1} \le \max\{\epsilon,\lambda(\epsilon)\}$, and
$0\le \epsilon_2 \le \epsilon$.
Dividing by $\epsilon$ and letting $\epsilon\to 0$,
as both $F_s$ and $F_x$ are continuous up to the boundary, we see
\begin{equation}
F_s(x,y,s) = \frac{1}{\eta'(0)}\bigl(\gamma'(0)YF (x,y,s)
-\lambda'(0)F_{x}(x,y,s)\bigr).
\end{equation}
Therefore, $F_s$ is smooth on the boundary.
\end{proof}

\begin{lemma} \label{lem:SmExtnCR}
Let $\Omega \subset \C^n \times \R$ be a bounded domain with smooth boundary.
Let $I\subset \mathbb{R}$ be an open interval. 
Suppose $(\partial \Omega)_I$ contains no CR singularities of
$\partial\Omega$, and each $(\partial \Omega)_{\{c\}}$ is connected for all
$c \in I$ and $f \colon (\partial \Omega)_I \to \C$
is a smooth function such that either
\begin{enumerate}[(i)]
\item $n > 1$ and $f$ is a CR function on ${(\partial \Omega)}_{I}$, or
\item $n = 1$ and for every $c \in I$ where $(\Omega)_{\{c\}}$
is nonempty, there exists a
continuous function on $(\overline{\Omega})_{\{c\}}$, holomorphic on
$(\Omega)_{\{c\}}$ extending $f|_{(\partial \Omega)_{\{c\}}}$.
\end{enumerate}
Then, there exists a smooth function $F \colon (\overline{\Omega})_I \to \C$,
CR on $(\Omega)_I$,  such that $F|_{(\partial \Omega)_{\{c\}}} = f$.

Furthermore if $\partial \Omega$ is real-analytic and $f$ is real-analytic,
then $F$ is real-analytic and it is a restriction of a function holomorphic
in a neighborhood of $(\overline{\Omega})_I$ in $\C^{n+1}$.
\end{lemma}

\begin{proof}
For $n > 1$, for each $c \in I$, using the standard Hartogs-Bochner
phenomenon, $f$ has a holomorphic extension to $(\Omega)_{\{c\}}$ given by
the Bochner-Martinelli integral, let us denote it by $F(z,c)$.  For $n=1$
we are given the extension.
In addition, this extension is smooth on
$(\overline{\Omega})_{\{c\}}$ (see \cite{Kytmanov}*{Theorem 7.1}).
So, $F$ is well-defined on $(\overline{\Omega})_I$, holomorphic along leaves
and so CR on $(\Omega)_I$, and $F\vert_{(\partial\Omega)_I}=f$.

For any $(z,s) \in (\Omega)_I$,
\begin{equation}
F(z,s) = \smashoperator[r]{\int\limits_{(\partial \Omega)_{\{s\}}}}
f(\zeta,\bar{\zeta},s)U(\zeta,\bar{\zeta},z,\bar{z}),
\end{equation}
where $U$ is the Bochner-Martinelli kernel.
Because $(\partial \Omega)_I$ contains no CR singular points, 
$(\partial \Omega)_I$ intersects the leaves $\{ s = c \}$
transversally.
Therefore there exists a small interval $J$ with $\overline{J} \subset I$
and $c \in J$, a bounded domain $V \subset \C^n$ with smooth boundary, and a
smooth diffeomorphism $\Phi \colon \overline{V} \times J \to
(\overline{\Omega})_J$ with $\overline{V} \times \{s\}$ going to
$(\overline{\Omega})_{\{s\}}$.  Without loss of generality we assume $I= J$,
in particular we assume $F$ is smooth on the compact set $(\partial
\Omega)_{\overline{I}}$.

We change the integral to be over the fixed domain $V$ and for
each fixed $s$ use $\Phi(\cdot,s)$ as a change of variables.  As the kernel
$U$ is smooth, $F(z,s)$ is smooth as a function on
$(\Omega)_I$.

Let us first consider the real-analytic case.
If $\partial \Omega$ is real-analytic, then $\Phi$ is
real-analytic.  So, as $U$ is real-analytic, $F(z,s)$ is also real-analytic.
Hence $F$ extends as a holomorphic function to a neighborhood of
$(\Omega)_I$ in $\C^{n+1}$, that is, replacing $s$ with a complex variable.  A real-analytic CR function on a
generic real-analytic CR submanifold $(\partial \Omega)_I$ has a unique
holomorphic extension to a neighborhood of $(\partial \Omega)_I$.  Putting the
two facts together we obtain $F$ extends to a holomorphic
function on a neighborhood of $(\overline{\Omega})_I$ in $\C^{n+1}$.  We are done in the real-analytic case,
so let us drop this requirement.

To show $F$ is smooth on $(\overline{\Omega})_I$, we use an iterative
approach.  We claim the following:
\begin{enumerate}[\quad (a)]
\item $F$ is continuous on $(\overline{\Omega})_I$, and
\item $F_{z_k}$, for any $1\leq k\leq n$, and $F_s$ are smooth on $(\partial \Omega)_I$.
\end{enumerate}
If the claims are proved, 
we may replace $f$ with $F_{z_k}\vert_{(\partial\Omega)_I}$ and
$F_{s}\vert_{(\partial\Omega)_I}$.  Hence, $F\in
C^1\bigl((\overline{\Omega})_I\bigr)$ and repeating the procedure gives $F\in
C^\infty\bigl((\overline{\Omega})_I\bigr)$.

The derivatives $F_{z_k}$ are uniformly bounded in $(\overline{\Omega})_I$ since they are
smooth up to the boundary on each leaf and since the normal derivative can be
bounded by tangential derivatives for holomorphic functions, we bound
$F_{z_k}$ by tangential derivatives of $f$. More precisely,
\begin{equation}\label{eq:zDerBd}
\begin{split}
\abs{F_{z_k} (z,s)} &\leq \sup\limits_{(\zeta,s)\in (\partial \Omega)_{\{s\}}} \abs{F_{z_k}(\zeta,s)}\\
&\le C \sup \left\{\abs{Xf(p)} \,:\, p\in (\partial
\Omega)_{\overline{I}}, X\in T_p\big((\partial \Omega)_{\overline{I}}\big),
\norm{X} = 1\right\} < \infty.
\end{split}
\end{equation}
To prove $F$ is continuous we must show $F$ is continuous at the boundary
points. Let $(z_0,s_0)\in(\partial\Omega)_I$. For
$(z,s)\in(\overline{\Omega})_I$ sufficiently close to $(z_0,s_0)$, we have
\begin{equation}\label{eq:ContF}
\begin{split}
\abs{F(z_0,s_0)-F(z,s)} & \leq
\abs{f(z_0,\bar{z}_0,s_0)-f(z',\bar{z}',s)}+\abs{f(z',\bar{z}',s)-F(z,s)} \\
& \leq \abs{f(z_0,\bar{z}_0,s_0)-f(z',\bar{z}',s)}+ M\snorm{z'-z}
\end{split}
\end{equation}
where $(z',s)$ is the normal projection of $(z,s)$ to the boundary along the
leaf $(\overline{\Omega})_{\{s\}}$, and
$M$ is the supremum of the $F_{z_k}$'s on $(\overline{\Omega})_{I}$.
We conclude $F$ is continuous at $(z_0,s_0)$ since $f$ is continuous on
$(\partial\Omega)_I$ and $(z',s)\to (z_0,s_0)$ as $(z,s)\to (z_0,s_0)$.

For $1\leq k\leq n$, $F_{z_k}$ extends to $(\partial\Omega)_I$ since it is
smooth up to the boundary when restricted to each leaf by Bochner-Martinelli.
To see these are smooth on $(\partial\Omega)_I$ we let $U$ be a
neighborhood of $(z,s)\in (\partial\Omega)_I$.
Since $(\partial\Omega)_I$ has no CR singularities, there are
smooth vector
fields $X_1,\ldots,X_{n-1}$ such that $T^{(1,0)}\bigl((\partial\Omega)_I\cap
U\bigr) =
\operatorname{span}\{X_1,\ldots,X_{n-1}\}$.  For $n=1$ we do not need these
vector fields.
As the leaves intersect $\partial\Omega$ transversally there is a
$1\leq k_0\leq n$ such that $\frac{\partial}{\partial z_{k_0}}$ is not tangent to
$\partial\Omega$ near $(z,s)$.
Without loss of generality suppose $k_0=1$, i.e.,
$\rho_{z_1}(z,\bar{z},s)\ne 0$ where $\rho\colon\C^n\times\R\to\R$ is a
defining function for $\Omega$. The `bad tangent direction' is given locally
by the vector field
\begin{equation}
X=\frac{\partial}{\partial z_1} - \frac{\rho_{z_1}}{\rho_{\bar{z}_1}}\cdot\frac{\partial}{\partial\bar{z}_1}.
\end{equation}
As $X$ differentiates along the leaves, we can differentiate $F$, which is
smooth up to the boundary along leaves, instead of $f$.
Notice $F_{z_1}\vert_{(\partial\Omega)_I}=XF\vert_{(\partial\Omega)_I} = Xf$
and $X_kF\vert_{(\partial\Omega)_I} = X_kf$ for $1\leq k\leq n-1$.
So $F_{z_1}$, $X_1F, \ldots, X_{n-1}F$ are smooth on $(\partial\Omega)_I$.
Making $U$ smaller if necessary, we have $\{\frac{\partial}{\partial z_1}$,
$X_1$, $\ldots$, $X_{n-1}\}$ is a smooth coordinate frame in each leaf in
$U$ and hence we get $F_{z_k}$ is smooth on $(\partial\Omega)_I$ for
$1\leq k\leq n$.

We now iterate.  The function $F_{z_k}$ is smooth on the boundary
$(\partial\Omega)_I$ and satisfies all the hypotheses.  We obtain
$F_{z_k}$ is continuous, and in fact all $z$
derivatives of $F$ of all orders are continuous up to the boundary.

In the $n=1$ case, we now have that $F$ satisfies the hypotheses of 
Lemma~\ref{lem:SmExtnHar} and so $F_s$ is smooth on the boundary.
Therefore $F_s$ also satisfies the hypotheses of the present lemma,
and we iterate as mentioned above to obtain $F$ is smooth.
Let us now suppose $n > 1$.

Define the vector field
\begin{equation}
Y = 
\frac{\partial}{\partial s} -
\sum_{k=1}^n
\left(\frac{\rho_s\rho_{z_k}}{\sum_j \rho_{z_j}\rho_{\bar{z}_j}}\right)
\frac{\partial}{\partial \bar{z}_k} .
\end{equation}
Because there are no CR singularities the denominator does not vanish and
$Y$ is a smooth vector field
defined globally on $(\partial \Omega)_I$.  Let $X$ be a CR vector field
defined near some point of $(\partial \Omega)_I$.  In particular $X$ has no
$\frac{\partial}{\partial s}$ component.  As $f$ is a CR function,
$Xf$ vanishes.  We look at the commutator $[X,Y]f$.  We do not claim this
commutator always vanishes, but we claim it vanishes for CR functions.
The commutator clearly does not contain a
$\frac{\partial}{\partial s}$ component as $X$ does not and the
corresponding coefficient in $Y$ is constant.  Therefore $[X,Y]$ is a
vector field which only differentiates along directions
tangent to the leaves $\{ s =
\text{constant} \}$.  Along the leaves, $F$ is smooth up to the
boundary and equals $f$.  Along each leaf, $F$ is also holomorphic and as
$[X,Y]$ only differentiates in the $\bar{z}$ directions, for any
fixed $s$, $[X,Y]F = 0$, and therefore $[X,Y]f = 0$.  In
particular, $Yf$ is a smooth CR function on $(\partial \Omega)_I$.
Therefore, by what we proved so far, $Yf$ extends to a continuous
function $G$ on $(\overline{\Omega})_I$, which is holomorphic along the leaves
and smooth up to the boundary along the leaves.  All the $z$ derivatives
of $G$ of all orders are also continuous up to the boundary.

Now suppose near some point $\rho_{\bar{z}_1} \not= 0$.  Take
\begin{equation}
\widetilde{Y} = 
\frac{\partial}{\partial s} -
\left( \frac{\rho_s}{\rho_{\bar{z}_1}}\right)
\frac{\partial}{\partial \bar{z}_1} .
\end{equation}
The vector field $(Y-\widetilde{Y})$ again only points
along the leaves and therefore we fix $s$ and compute
$(Y-\widetilde{Y})f = 0$ using $F$.
So $\widetilde{Y}f = Yf$ wherever $\widetilde{Y}$ is defined.

Let $s_0 \in I$ be fixed.  As $(\Omega)_{\{s_0\}}$ is smoothly bounded,
there exists a complex line in $\{ s = s_0 \}$ that is
outside of $(\Omega)_{\{s_0\}}$, and touches
$(\overline{\Omega})_{\{s_0\}}$ at a single point.
Write $z = (z_1,z')$.  
Without loss of generality suppose this line is $\{ z' = 0, s=s_0\}$.
For $z_0'$ near zero and $s$ near $s_0$, the intersection
$(\Omega)_{\{s\}} \cap \{ z' = z_0' \}$
is either a single point, empty, or
a bounded domain (in $\C$) with smooth boundary.
Let $J$ be a small interval
around $s_0$.  Suppose also that $J$ is small enough and a fixed
$z_0'$ is such that we obtain
$(\Omega)_{\{s\}} \cap \{ z' = z_0' \}$ are
bounded domains with smooth boundary for all $s \in J$.
In particular we are in the setup for $n=1$.  We know 
$\frac{\partial}{\partial z_1}$ is not tangent to the boundary of
$(\Omega)_{J} \cap \{ z' = z_0' \}$.  On this fixed slice,
we apply the $n=1$ result, so along this slice $F_s$ is smooth up
to the boundary.  Note $\widetilde{Y}$ points along this slice and
so $F_s$ on the boundary is equal to $\widetilde{Y}f$, which is equal to
$Yf$.  As we can repeat this argument for a small open set of $z_0'$ we
obtain for the $s \in J$, $G$ is equal to $F_s$ on an open set, and
therefore by the identity theorem $G = F_s$.  Thus $F_s$ extends
continuously to the boundary and is smooth on the boundary.

We now have $F_{z_k}$ and $F_s$ satisfy the hypotheses of the lemma,
that is we have proved claim (b), and we iterate to obtain $F$
is $C^\infty\bigl((\overline{\Omega})_I\bigr)$.
\end{proof}


\section{Polynomial case}
\label{section:poly}

Let us prove the polynomial case first; we wish to extend
polynomials in the model case where $E = 0$.
We begin with $n=1$.

\begin{lemma} \label{lem:PolyExtn1}
Suppose $M, H \subset \C^2$ are defined by
\begin{equation} 
M: w=\sabs{z}^2 + \lambda (z^2 + \bar{z}^2) ,\qquad
H:  \begin{cases}
\Re w \geq \sabs{z}^2 + \lambda (z^2 + \bar{z}^2) ,\\
\Im w = 0.
\end{cases}
\end{equation}
where $0 \leq \lambda < \frac{1}{2}$. 
Suppose $f \colon M \to \C$ is a function such that if
$M$ is parametrized by $z$, $f(z,\bar{z})$ is a polynomial.
Further suppose 
for every $c > 0$, there exists a 
continuous function on $H \cap \{ w = c \}$, holomorphic on $(H \setminus M)
\cap \{ w = c \}$ extending $f|_{M \cap \{ w = c \}}$.

Then there exists a holomorphic polynomial $P(z,w)$
such that $P|_M = f$.
Furthermore, if $f$ is homogeneous of degree $d$, then 
$P$ is of weighted degree $d$, that is 
\begin{equation}
P(z,w) = \sum\limits_{j + 2k = d}\, c_{j,k}\,z^j\, w^k.
\end{equation}
\end{lemma}

\begin{proof} 
Let $\rho(z,\bar{z}) = \sabs{z}^2 + \lambda (z^2 + \bar{z}^2)$.
Write $f(z,\bar{z})$ for the value of $f$ on $M$ at
$\bigl(z,\rho(z,\bar{z})\bigr)$.

By a classical result (probably due to Lam\'e) the Dirichlet
problem on an ellipse with polynomial data has a polynomial solution.
For a wonderful exposition and a beautiful short proof (probably due to
Fischer), see \cite{KhavinsonLundberg}.  Therefore, for every $c > 0$,
there exists a polynomial $P_c(z)$ such that $P_c(z) = f(z,\bar{z})$
when $\rho(z,\bar{z}) = c$.
We need to show the dependence on $c$ is also
polynomial.
We will treat $z$ and $\bar{z}$ as separate variables from now on.

First suppose $\lambda = 0$.  Consider the rational function
$P(z,w) = f(z,\frac{w}{z})$.  For all fixed $c > 0$ we get 
$P_c(z) = P(z,c)$ on a circle and hence everywhere
for $z \not= 0$.
If $f(z,\bar{z}) = \sum a_{jk} z^j \bar{z}^k$, then
\begin{equation}
P_c(z) = f\Bigl(z,\frac{c}{z}\Bigr) = \sum a_{jk} z^{j-k} c^k .
\end{equation}
This can only be a polynomial for all $c > 0$ if $a_{jk} = 0$ whenever $j <
k$.  Therefore $P(z,w)$ is the polynomial we are looking for.

Now suppose $0 < \lambda < \frac{1}{2}$.
The map
\begin{equation}
(z,\bar{z}) \overset{\Phi}{\mapsto} \left(z,-\frac{1}{\lambda}z-\bar{z}\right)
\end{equation}
is an involution fixing
$\rho(z,\bar{z}) = z \bar{z} + \lambda (z^2 + \bar{z}^2)$.

Take $z \not= 0$ and let $c$ be such that $\rho(z,\bar{z}) = c$.  We have
$P_c(z) = f(z,\bar{z})$ when $\rho(z,\bar{z}) = c$, and 
as everything is a polynomial in $z$ and $\bar{z}$ we obtain
\begin{equation}
P_c(z) - f(z,\bar{z}) = q(z,\bar{z}) \bigl(\rho(z,\bar{z}) - c\bigr) ,
\end{equation}
for some polynomial $q$.  In particular applying the involution we get
\begin{equation}
P_c(z) - f(z,-z/\lambda-\bar{z}) = q(z,-z/\lambda-\bar{z}) \bigl(\rho(z,\bar{z}) - c\bigr) ,
\end{equation}
and so $f$ is invariant under the involution.

A well-known theorem of
Noether (see e.g.\ \cite{Sturmfels}*{Theorem 2.1.4})
says that the algebra of polynomials invariant under
$\Phi$ is generated by
$m + m \circ \Phi$, where $m(z,\bar{z})$ varies over monomials of degree at most $2$.
A short
calculation shows this algebra is generated by $z$ and
$\sabs{z}^2 + \lambda (z^2 + \bar{z}^2)$.  We obtain our polynomial
$P$.

The claim about degree follows at once.
\end{proof}

The following proposition is surely classical and is useful for
extending the one dimensional result to higher dimensions.

\begin{prop} \label{prop:formalpoly}
Suppose $F(z) = F(z_1,\ldots,z_n)$ is a formal power series in $n$
variables.  Let $\epsilon S^{2n-1} \subset \C^n$ be the sphere of radius
$\epsilon$ and suppose $U \subset \epsilon S^{2n-1}$ is an open subset.
If for every $z \in U$
\begin{equation}
P_z(t) = F(tz)
\end{equation}
is a polynomial, then $F$ is a polynomial.  Furthermore,
\begin{equation}
\deg F = \max_{z \in U} \deg P_z .
\end{equation}
\end{prop}

\begin{proof}
Write
\begin{equation}
F(z) = \sum_{j=0}^\infty f_j(z) ,
\end{equation}
where $f_j$ is a homogeneous polynomial of degree $j$.  If $f_j$ is not
identically zero then $f_j$ is nonzero on a second category subset of 
$U$.  Hence there exists a $z_0 \in U$ such that
for all $j$, $f_j(z_0) = 0$ if and only if $f_j \equiv 0$.  Take
\begin{equation}
P_{z_0}(t) = F(tz_0) = \sum_{j=0}^\infty t^j f_j(z_0) .
\end{equation}
As $P_{z_0}(t)$ is a polynomial of degree $d$, then $f_j \equiv 0$ for
all $j > d$, that is, $F$ is a polynomial of degree $d$, and the degree
is maximized by $P_{z_0}(t)$.
\end{proof}

\begin{lemma} \label{lem:PolyExtn}
Suppose $M, H \subset \C^{n+1}$, $n > 1$, are defined by
\begin{equation} 
M : 
w = \sum_{j=1}^n \bigl(\sabs{z_j}^2 + \lambda_j (z_j^2 + \bar{z}_j^2) \bigr),
\quad
H : 
\begin{cases}
\Re w \geq \sum_{j=1}^n \bigl(\sabs{z_j}^2 + \lambda_j (z_j^2 + \bar{z}_j^2) \bigr),
\\
\Im w = 0 .
\end{cases}
\end{equation}
where $0 \leq \lambda_j < \frac{1}{2}$. 
Suppose $f \colon M \to \C$ is a CR function such that if
$M$ is parametrized by $z$, $f(z,\bar{z})$ is a polynomial.

Then there exists a holomorphic polynomial $P(z,w)$
such that $P|_M = f$.
Furthermore, if $f$ is homogeneous of degree $d$, then 
$P$ is of weighted degree $d$, that is,
\begin{equation}
P(z,w) = \sum\limits_{\sabs{\alpha} + 2k = d}\, c_{\alpha,k}\,z^\alpha\, w^k.
\end{equation}
\end{lemma}

\begin{proof}
By Lemma~\ref{lem:SmExtnCR}
we obtain an extension $P(z,w)$, holomorphic in some neighborhood of $H \setminus \{ 0 \}$.
In particular it is holomorphic on a neighborhood of the point $(0,\ldots,0,1)$.

Pick a $z \in \C^n$ with $\snorm{z}=1$.  Then we use the mapping
from $\C^2$ to $\C^{n+1}$ given by
\begin{equation}
(\xi,w) \mapsto (\xi z, w)
\end{equation}
and the one dimensional case to find the mapping $P$ extends to a
polynomial along every plane that contains the $w$-axis.  In particular, $P$
is a polynomial along every complex line through the point $(0,\ldots,0,1)$,
and by Proposition~\ref{prop:formalpoly} $P$ is a polynomial.

The degree claim follows easily: If we plug in the expression for $w$
into a monomial in $z$ and $w$ of weighted degree $k$ we get a homogeneous
polynomial 
of degree $k$ in $z$ and $\bar{z}$.  So taking the weighted degree $d$ part
of $P$ would have sufficed above.  It is easy to see $P$ must be unique
and therefore it only has a weighted degree $d$ part.
\end{proof}

The polynomial case only holds for the quadratic models.  In general
even if $M$ is given by a polynomial, a polynomial $f$ does not mean $F$ is
a polynomial.

\begin{example}
Take $M$ to be given by $w = \sabs{z}^2+\sabs{z}^4$, and $H$ given as usual.
The function $F(z,w) = \sqrt{1+4w}$ (some branch of the root) is a
holomorphic function in a neighborhood of the origin.  $F$ is clearly not
a polynomial, but on $M$ parametrized by $z$ the function becomes
$1+2\sabs{z}^2$, a polynomial.
\end{example}


\section{Smooth case} \label{section:smooth}

In this section we prove the assertion for the smooth case in Theorem~\ref{thm:mainlocal}.
We begin by describing the coordinates we use and then collect some
crucial building blocks of our proof organized as lemmas and propositions.
Recall we consider the following locally near the origin:
\begin{equation}
\begin{aligned}
M&: w=Q(z,\bar{z}) + E(z,\bar{z}),
&M^{quad}&: w=Q(z,\bar{z}),\\
H&:  \begin{cases}\Re w \geq Q(z,\bar{z}) + E(z,\bar{z}) , \\
\Im w = 0,\end{cases}
&H^{quad}&:  \begin{cases}\Re w \geq Q(z,\bar{z}), \\
\Im w = 0,\end{cases}
\end{aligned}
\end{equation}
where $Q(z,\bar{z}) = \sum\limits_{j=1}^{n} \bigl(\sabs{z_j}^2 + \lambda_j
(z_j^2 +\bar{z}_j^2)\bigr)$, $0 \leq \lambda_j < \frac{1}{2}$, and $E$ is
$O(3)$ and real-valued.

\begin{lemma}\label{lem:OrderK}
Let $U$, $M$, $H$, $f$ be as in Theorem~\ref{thm:mainlocal}.  Parametrizing
$M$ by $z$ as usual, suppose $f$ is $O(k)$ at $0$.
Then the $k$th order homogeneous part $f_k$ of $f$ satisfies the hypotheses
of Theorem~\ref{thm:mainlocal} for $M^{quad}$.
In particular, there exists a holomorphic polynomial $P(z,w)$ in $\C^{n+1}$
of weighted degree $k$ such that $(f-P)\vert_M$ is $O(k+1)$ at $0$.
\end{lemma}

\begin{proof}
Let us first consider the $n=1$ case.

For $s \geq 0$,
we wish to parametrize $M\cap\{w=s\}$ by $r\varphi(r,\theta)e^{i\theta}$ where $s=r^2$, and $\varphi$ is real-valued,
i.e.,
\begin{equation}
r^2 = r^2 \varphi^2 + 2 \lambda r^2 \varphi^2 \cos \theta
+ E(
r\varphi e^{i\theta},
r\varphi e^{-i\theta}) .
\end{equation}
Since $E$ is $O(3)$ we may divide by $r^2$ to get
\begin{equation}
1 = \varphi^2 + 2 \lambda \varphi^2 \cos \theta
+ \frac{1}{r^2} E(
r\varphi e^{i\theta},
r\varphi e^{-i\theta}) .
\end{equation}
We may now apply the
implicit function theorem to this equation to guarantee that such a
smooth $\varphi$ exists for a neighborhood of $\{ 0 \} \times [0,2\pi]$,
that is, for $r$ in a small interval around zero and all $\theta$.

By hypothesis, $f$ has a holomorphic
extension on each leaf.  The hypothesis is equivalent to saying for each $r>0$ and
$\ell \geq 0$,
\begin{equation}
\smashoperator[r]{\int\limits_{M\cap\{w = r^2\}}}
f\left(\zeta,\bar{\zeta}\right)\zeta^\ell \, d\zeta = 0 .
\end{equation}
To see why this condition is equivalent to the hypothesis
(it is standard for a circle),
simply extend $f$ to $H \cap \{ w = r^2 \}$ as a harmonic function
$h(z)-g(\bar{z})$,
then apply Green's theorem, to obtain
$0 = \int_{H\cap \{w=r^2\}} g'(\bar{\zeta}) \zeta^\ell d\zeta \wedge
d\bar{\zeta}$, and by the density of holomorphic polynomials (Mergelyan's
theorem) we get that $g'(\bar{z}) = 0$.  Hence $f$ extends holomorphically.
The integral condition
is equivalent to
\begin{equation}
\int\limits_0^{2\pi}
f\left(r\varphi e^{i\theta},r\varphi e^{-i\theta}\right)
\varphi^\ell  (\varphi_\theta+i\varphi)e^{i(\ell +1)\theta}\, d\theta = 0.
\end{equation}
Let $f_k$ be the $k$th order part of $f$.  The above expression is
divisible by $r^k$.  Dividing and letting $r \to 0$, and then multiplying
through by $r^k$ again we obtain
\begin{equation}
\int\limits_0^{2\pi} 
f_k\left(r\varphi e^{i\theta},r\varphi e^{-i\theta}\right)
\varphi^\ell(0,\theta)
\bigl(\varphi_\theta(0,\theta)+i\varphi(0,\theta)\bigr)
e^{i(\ell+1)\theta}\,d\theta
= 0.
\end{equation}
This shows $f_k$ satisfies the conditions of Lemma~\ref{lem:PolyExtn1} on
$M^{quad}$.  Therefore, there exists a polynomial $P$ in $z$ and $w$ such
that $P = f_k$ on $M^{quad}$, that is $P\bigl(z,Q(z,\bar{z})\bigr) =
f_k(z,\bar{z})$.  Therefore
$f(z,\bar{z}) - P\bigl(z,Q(z,\bar{z})+E(z,\bar{z})\bigr)$ is of order $k+1$.

For $n > 1$ we have a local condition, the function is CR.
Write $w = \rho(z,\bar{z})$ for $M$.  Then the CR vector fields on $M$ in
the intrinsic $z$ coordinates can be written as
\begin{equation}
X =
\rho_{\bar{z}_j} \frac{\partial}{\partial \bar{z}_\ell}
-
\rho_{\bar{z}_\ell} \frac{\partial}{\partial \bar{z}_j} .
\end{equation}
For any such $X$ we have $Xf = 0$. Furthermore,
\begin{equation}
\lim_{r \to 0}
\frac{1}{r^k} Xf (rz,r\bar{z}) = 
\widetilde{X}f_k(z,\bar{z}) ,
\end{equation}
where
\begin{equation}
\widetilde{X} =
Q_{\bar{z}_j} \frac{\partial}{\partial \bar{z}_\ell}
-
Q_{\bar{z}_\ell} \frac{\partial}{\partial \bar{z}_j}
\end{equation}
are precisely the CR vector fields on $M^{quad}$.
Therefore, just as above we find $f_k$ satisfies the hypothesis of
Theorem~\ref{thm:mainlocal} for $M^{quad}$.
Apply Lemma \ref{lem:PolyExtn} to obtain a $P$.  We finish the proof
exactly as for $n=1$ case.
\end{proof}

\begin{lemma}\label{lem:ExtnConts}
Let $U$, $M$, $H$, $f$ be as in Theorem~\ref{thm:mainlocal}.
If $f\in C^\infty(M)$, then $F\in C(H)$.
\end{lemma}
\begin{proof}

By Lemma~\ref{lem:SmExtnCR}, $F\in C(H\setminus\{0\})$. The continuity of $F$ at the origin follows from the maximum principle: For $(z,s)\in H$ we have
\[\abs{F(z,s)-f(0)} \le \sup\Big\{\abs{f(\zeta,\bar{\zeta})-f(0)}\,:\,(\zeta,s)\in M\Big\}.\]

As $s\to 0$, $M\ni(\zeta,s) \to 0$.
\end{proof}

As the derivatives $F_{z_j}$ and $F_s$ extend smoothly to $H \setminus \{ 0 \}$ we
must next prove they extend through the origin smoothly.  First we show
their restrictions to $M$ extend smoothly through the origin.  As the
extension is unique, we abuse notation
slightly and call the extensions $F_{z_j}$ and $F_s$ and think of them as a
functions on $H$ or $M$ as needed.

\begin{lemma}\label{lem:DerFSm}
Let $U$, $M$, $H$, $f$ be as in Theorem~\ref{thm:mainlocal}.
If $f\in C^\infty(M)$, then $F_{z_j},F_s\in C^\infty(M)$, for $1\leq j\leq n$.
\end{lemma}

\begin{proof}
Let $w = \rho(z,\bar{z})$ define $M$, then write $\xi_j = \rho_{\bar{z}_j} =
z_j + 2 \lambda_j \bar{z}_j + \cdots$.
We start by taking derivatives outside the origin:
\begin{equation}
f_{\bar{z}_j} = F_s \xi_j .
\end{equation}
Lemma \ref{lem:OrderK} says, for any order $m$, we may write $f$ as
\begin{equation}
f(z,\bar{z}) = P(z,\rho) + R(z,\bar{z}) ,
\end{equation}
where $P(z,w)$ is a polynomial of degree $m$
and $R$ is $O(m+1)$.  Therefore,
\begin{equation}
f_{\bar{z}_j} = P_w(z,\rho) \xi_j + R_{\bar{z}_j}(z,\bar{z}) .
\end{equation}
In particular, no lower order terms are an obstruction to the division by
$\xi_j$.
The variables $\xi,\bar{\xi}$ are a smooth change of variables, we therefore
just think of everything in terms of $\xi$ and $\bar{\xi}$.

Using the real part of $\xi_j$ as a variable we apply the Malgrange-Mather
division theorem~\cite{Malgrange}*{Chapter V}.
We obtain smooth functions $q$ and
$r$ such that
\begin{equation} \label{eq:fjdiv}
f_{\bar{z}_j} = q \xi_j + r ,
\end{equation}
where $r$ does not depend on the real part of $\xi_j$.
By the above argument, $r$ cannot have any finite order terms, otherwise the
finite order part could be put (modulo higher order terms) into something
divisible by $\xi_j$ and hence would depend on the real part of $\xi_j$.

Now, consider the ideal $I$ generated by $\xi_j$.
To show $f_{\bar{z}_j}\in I$,  it suffices, by a theorem of Malgrange \cite{Malgrange}*{Theorem 1.1' in Chapter VI}, to show the Taylor series of $f_{\bar{z}_j}$ at each point belongs formally to $I$.
This is clear outside the origin as $f_{\bar{z}_j}$ is in fact divisible by $\xi_j$ there.
At the origin, \eqref{eq:fjdiv} along with the fact that $r$ vanishes to
infinite order shows that formally $f_{\bar{z}_j}$ is in $I$ because its
Taylor series is equal to the Taylor series of a function that is in the ideal, namely $q\xi_j$.
Since $f_{\bar{z}_j}$ is in $I$, it is divisible by $\xi_j$ and therefore $F_s|_M$ extends smoothly through the origin.

Now that $F_s|_M$ is smooth we write
\begin{equation}
f_{z_j} = F_{z_j} + F_s \bar{\xi}_j .
\end{equation}
As $f_{z_j}$ and $F_s \bar{\xi}_j$ are smooth on $M$, then $F_{z_j}$ 
extends smoothly through the origin as well.
\end{proof}

We now prove the smooth part of Theorems~\ref{thm:mainlocal}
and \ref{thm:mainglobal}. 
For reader convenience we state the two results as two lemmas.

\begin{lemma} \label{lem:localsmext}
Let $U$, $M$, $H$, $f$ be as in Theorem~\ref{thm:mainlocal}.
Then there exists a function $F \in C^\infty(H)$ such that
$F$ is CR on $H \setminus M$ and $F|_M = f$.
Furthermore, $F$ has a formal power series at 0 in $z$ and $w$.
\end{lemma}

\begin{proof}
Lemma~\ref{lem:SmExtnCR} shows $F\in C^\infty(H \setminus \{ 0 \})$.  We
first show $F \in C^{1}(H)$.  We know from Lemma~\ref{lem:ExtnConts}
that $F \in C(H)$.  We also know from Lemma~\ref{lem:DerFSm} that the
derivatives $F_{z_j}$ and $F_s$ extend to functions which are smooth on $M$
(and smooth on $H \setminus \{ 0 \}$ of course).  Clearly $F_{z_j}|_M$ and
$F_s|_M$ satisfy the hypotheses of Theorem~\ref{thm:mainlocal}.  By
Lemma~\ref{lem:ExtnConts} we obtain their continuity and therefore $F$
is $C^1$.  By iterating this procedure we obtain $F \in C^\infty(H)$.
\end{proof}

\begin{lemma} \label{lem:globalsmext}
Let $\Omega$ and $f$ be as in Theorem~\ref{thm:mainglobal}.
Then there exists a function $F \in C^\infty(\overline{\Omega})$ such that
$F$ is CR on $\Omega$ and $F|_{\partial \Omega} = f$.
\end{lemma}

\begin{proof}
By Proposition~\ref{prop:topology} $\partial \Omega$ is homeomorphic to a
sphere with two CR singular points, say $p_1$ and $p_2$.
Lemma~\ref{lem:SmExtnCR} shows there exists $F\in C^\infty(\overline{\Omega} \setminus \{
p_1 , p_2 \})$ extending $f$.  Then by Lemma~\ref{lem:localsmext} we obtain
$F \in C^\infty(\overline{\Omega})$
\end{proof}


\section{The real-analytic case}
\label{section:realanal}

We are ready to prove the real-analytic part of the results:
if the boundary and $f$
are real-analytic then we get a holomorphic extension.  First we
prove the case $n=1$.

\begin{lemma}
Suppose $M, H \subset U = \{ (z,w) \in \C \times \C : 
\sabs{z} < \delta_z, \sabs{w} < \delta_w \}$ are defined by
\begin{equation}
M: w= \sabs{z}^2 + \lambda (z^2 +\bar{z}^2) + E(z,\bar{z}), \qquad 
H:  \begin{cases}\Re w \geq \sabs{z}^2 + \lambda (z^2 +\bar{z}^2)  + E(z,\bar{z}) ,\\
\Im w = 0,\end{cases}
\end{equation}
where $0 \leq \lambda < \frac{1}{2}$, 
$E$ is $O(3)$, real-valued, real-analytic,
and $\delta_z,\delta_w > 0$ are small enough as defined in the introduction.
Suppose $f \colon M \to \C$ is real-analytic and
for every $0< c < \delta_w$, there exists a 
continuous function on $H \cap \{ w = c \}$ that is holomorphic on $(H \setminus M)
\cap \{ w = c \}$ and extends $f|_{M \cap \{ w = c \}}$.

Then there exists a holomorphic function $F$ defined in a
neighborhood of $H$ in $\C^2$ such that $F|_M = f$.
\end{lemma}

\begin{proof}
Let $\rho(z,\bar{z}) = \sabs{z}^2 + \lambda (z^2 + \bar{z}^2) + E(z,\bar{z})$.
Write $f(z,\bar{z})$ for the value of $f$ on $M$ at
$\bigl(z,\rho(z,\bar{z})\bigr)$ as usual.
By Lemma~\ref{lem:SmExtnCR}
we obtain an extension $F(z,w)$, holomorphic in some neighborhood of $H \setminus \{ 0 \}$.

As everything in sight is real-analytic, we locally
complexify and treat $z$ and $\bar{z}$ as independent variables.
The proof splits into two cases.

\textbf{Case 1:} Either $\lambda > 0$ or $E(z,0) \not\equiv 0$ (i.e.\
$\rho(z,0) \not\equiv 0$).

Let $G(z,\xi) = \bigl(z,\rho(z,\xi) \bigr)$.  By our assumption
$G^{-1}(0) = \{ 0 \}$, and so
$G$ is a finite holomorphic map at the origin,
which is generically $k$-to-1 for some for some $k$.
The function defined at a generic point by
\begin{equation}
\widetilde{F}(z,w) :=
\frac{1}{k}\sum_{(\zeta,\xi) \in G^{-1}(z,w)}
f(\zeta,\xi)
\end{equation}
is well-defined and holomorphic on a neighborhood of the origin;
it is a symmetric function of a $k$-valued function.  See
e.g.\ \cite{Whitney:book}*{Lemma 8A in chapter 1}.
Without loss of generality suppose
$\delta_z,\delta_w$ are small enough so that $\widetilde{F}$ is defined in
$U$.

Outside of the origin, $(F \circ G)(z,\bar{z}) = f(z,\bar{z})$.
The equality also holds formally at the origin, that is up to arbitrary
order.  For a fixed $z$ and $w$ and any solution $\xi$ of $w = \rho(z,\xi)$,
$(F \circ G)(z,\bar{z}) = (F \circ G)(z,\xi)$ formally at the origin.  As $f$ is
real-analytic we then get $f(z,\xi) = f(z,\bar{z})$ (not only formally).
Let $X$ be a small punctured neighborhood of the origin in the set
$\{ \xi = \bar{z} \}$ in the $(z,\xi)$-space.
The image $G(X)$ is a generic, totally-real submanifold
at most points as $G$ is a local biholomorphism outside a complex
subvariety.
On $G(X)$ we get $F = \widetilde{F}$, and therefore 
$F = \widetilde{F}$ on a neighborhood of $H \setminus \{ 0 \}$.
We are done.

\textbf{Case 2.} Suppose $\lambda = 0$ and $E(z,0) \equiv 0$.

The manifold $M$
has infinite Moser invariant;
successively taking changes of coordinates sending
$z$ to $z + a z^j w^k$ we can make $E$ vanish up to arbitrary order.
Moser~\cite{Moser85} proved there exists a local
biholomorphic change of variables near zero such that $M$ is given by
\begin{equation}
w = z\bar{z} ,
\end{equation}
in these new coordinates.  The change of variables must leave the set
$\Im w = 0$ invariant, and the leaves are still given by $\{ w = c \}$, although
the constants have changed.
Therefore the setup of the problem is the same.
If
\begin{equation}
f(z,\bar{z}) =
\sum c_{k,j} z^k {\bar{z}}^j ,
\end{equation}
Cauchy estimates dictate
$\sabs{c_{k,j}} \leq \frac{M}{\epsilon^{k+j}}$ for some $\epsilon > 0$.
Using the smooth case,
$F(z,s)$ has a
formal power series at the origin in $z$ and $s$,
and so $c_{k,j} = 0$ if $k < j$.
Write $d_{k,j} = c_{k+j,j}$.  Then
\begin{equation}
\sabs{d_{k,j}} = \sabs{c_{k+j,j}} \leq \frac{M}{\epsilon^{k} \epsilon^{2j}} .
\end{equation}
Therefore 
\begin{equation}
F(z,s) = \sum_{j,k} d_{k,j} z^ks^j
\end{equation}
converges.
\end{proof}

We now prove the real-analytic part of Theorems~\ref{thm:mainlocal} (for $n
> 1$)
and \ref{thm:mainglobal}.
For reader convenience we state the two results as two lemmas.

\begin{lemma} \label{lem:localraext}
Let $U$, $M$, $H$, $f$ be as in Theorem~\ref{thm:mainlocal}.
Suppose $n > 1$, and $E$ and $f$ are real-analytic.
Then there exists a holomorphic function $F$ defined in a neighborhood
of $H$ in $\C^{n+1}$ such that $F|_M = f$.
\end{lemma}

\begin{proof}
By Lemma~\ref{lem:SmExtnCR} 
we obtain an extension $F(z,s)$, holomorphic in some neighborhood of $H \setminus \{ 0 \}$.

Via the smooth extension (the smooth case of Theorem~\ref{thm:mainlocal} and in particular Lemma~\ref{lem:OrderK}) 
there is a smooth extension to $H$ with the formal Taylor series at the
origin written as a formal power series in $z$ and $s$.  We must show
this series converges and converges to a holomorphic function that
agrees with the above extension on a neighborhood of $H \setminus \{ 0 \}$.

A formal power series at the origin
which converges when restricted to every line through the origin converges
via a standard Baire category argument
(see e.g.\ \cite{BER:book}*{Theorem 5.5.30}).
We show it converges on every complex-2-dimensional plane
containing the $w$-axis.

Let $v \in \C^n$ with $\snorm{v} = 1$ be given.  Take the
complex-2-dimensional plane $P$ parametrized
by $(\xi,w) \mapsto (\xi v, w)$.  On $P$ apply the $n=1$ result
to find the function extends.  Furthermore the formal power series
for $F$ at $0$ in $z$ and $s$ restricted to $P$ is the power series for the 
extension on $P$.  Therefore on $P$ we get 1) the formal power series
for $F$ in $z$ and $s$ converges and 2) it equals the extension on some open
set of $H \cap P$.  By taking all possible $v$ we get the power series
for $F$ in $z$ and $s$ converges and equals to the extension on some open
subset of $H$.  We are done.
\end{proof}

\begin{lemma}
Let $\Omega$ and $f$ be as in Theorem~\ref{thm:mainglobal} and suppose
$\partial \Omega$ and $f$ are both real-analytic.
Then there exists a holomorphic function $F$ defined in a neighborhood
of $\overline{\Omega}$ in $\C^{n+1}$ such that $F|_{\partial \Omega} = f$.
\end{lemma}

\begin{proof}
If $S \subset \partial \Omega$ is the
set of CR singular points, then via Lemma~\ref{lem:SmExtnCR} and the local
extension at CR points we obtain a holomorphic function on some
neighborhood of $\overline{\Omega} \setminus S$.  Then via
Lemma~\ref{lem:localraext} 
the function extends to be holomorphic in some neighborhood of
$\overline{\Omega}$ in $\C^{n+1}$.
\end{proof}


\section{Degenerate holomorphically flat CR singularities}
\label{section:degenerate}

We close by making a few remarks on the situation when the CR singularity is
degenerate, but where we still keep the positivity aspect of the
ellipticity.
We focus on the local situation.
Suppose we have the following near the origin in $\C^{n+1}$.
\begin{equation} \label{eq:degenMH}
M : 
w = \rho(z,\bar{z}),
\quad
H : 
\begin{cases}
\Re w \geq \rho(z,\bar{z}),
\\
\Im w = 0 ,
\end{cases}
\end{equation}
where $\rho$ is smooth, real-valued, and positive near the origin (except at the origin where  $\rho(0)=0$).
The condition on positivity implies that
for each small fixed $s \in \R$
the intersection $M \cap \{ w = s \}$ is either empty or compact such
that the diameter of such sets goes to zero as $s \to 0$.  Finally suppose
$0$ is the only CR singularity of $M$.

By Lemma~\ref{lem:SmExtnCR} we obtain the extension smoothly up to
the CR points of $M$.
We also show that the $z_k$ derivatives are bounded.
That is, in the degenerate case,
the best we can obtain is the following proposition.

\begin{prop}
Suppose $H$ and $M$ are closed submanifolds of $U = \{(z,w) \in \C^n \times \C : 
\snorm{z} < \delta_z, \sabs{w} < \delta_w \}$ given by
\eqref{eq:degenMH}, and
$\delta_z,\delta_w > 0$ are small enough as defined in the introduction.

Suppose $f \colon M \to \C$ is smooth and either
\begin{enumerate}[(i)]
\item $n > 1$ and $f$ is a CR function on $M_{CR}$ (the CR points of $M$), or
\item $n = 1$ and for every $0< c < \delta_w$, there exists a 
continuous function on $H \cap \{ w = c \}$, holomorphic on $(H \setminus M)
\cap \{ w = c \}$ extending $f|_{M \cap \{ w = c \}}$
\end{enumerate}

Then there exists a function $F \in C(H) \cap C^\infty(H \setminus \{ 0 \})$ such that
$F$ is CR on $H \setminus M$ and $F|_M = f$.
Furthermore the derivatives $\frac{\partial F}{\partial z_k}$, $1\leq k\leq n$, are locally bounded at $0$.
\end{prop}

\begin{proof}
Because 0 is the only CR singularity of $M$, $d\rho$ is only zero
at the origin.  Therefore, for small $c > 0$
$M \cap \{ w = c \}$ are smooth compact submanifolds.
Furthermore,
$M \cap \{ w = c \}$ are connected as they are level sets of $\rho$,
and $\rho$ has only a single critical point.

The existence and smoothness of $F$ outside the origin
follows from Lemma~\ref{lem:SmExtnCR} as mentioned
above.  The proof of Lemma~\ref{lem:ExtnConts} works in this setting giving
us the continuity of the extension at the origin.

To show the boundedness of the $\frac{\partial F}{\partial z_k}$ near $0$, we use an approach similar to the one employed in~\eqref{eq:zDerBd}.
Since our context is quite concrete here, we are able to make this more explicit and direct. 
For $s>0$ and $1\le k\le n$, the following hold on $M\setminus\{0\}$ near the origin:
\begin{equation}
f_{z_k} = F_{z_k} + F_s\rho_{z_k}\quad\text{and}\quad
f_{\bar{z}_k}=F_s\rho_{\bar{z}_k} ,
\end{equation}
which gives us
\begin{equation}
\abs{F_{z_k}}  \le \abs{f_{z_k}} + \abs{F_s\rho_{z_k}}  = \abs{f_{z_k}} +
\abs{F_s\rho_{\bar{z}_k}} = \abs{f_{z_k}} + \abs{f_{\bar{z}_k}}\ \text{ on }
M \setminus \{ 0 \}.
\end{equation}
Since $F_{z_k}$'s extend smoothly up to the boundary on each leaf of $H\setminus\{0\}$, we have by the maximum principle
\begin{equation}
\abs{F_{z_k}(z,s)} \le \sup\limits_{(\zeta,s)\in M \setminus \{0\}}\, \abs{F_{z_k}(\zeta,s)} \le \sup\limits_{(\zeta,s)\in M}\, \abs{f_{z_k}(\zeta,\bar{\zeta})} +\abs{f_{\bar{z}_k}(\zeta,\bar{\zeta})}
\end{equation}
for $(z,s)\in H\setminus\{0\}$.
The conclusion follows by noting that $f_{z_k}$ and $f_{\bar{z}_k}$ are smooth on $M$ and hence bounded near the origin.
\end{proof}

This proposition along with Example~\ref{eg:NonDegen} shows that at degenerate CR singularities
it is the derivative in the normal direction (the $s$ direction) that might
blow up, but the function $F$ itself is at least continuous with bounded $z$
derivatives.  The crux
of the extension, where the nondegeneracy is necessary, is in bounding the $s$ derivative.


\def\MR#1{\relax\ifhmode\unskip\spacefactor3000 \space\fi%
  \href{http://www.ams.org/mathscinet-getitem?mr=#1}{MR#1}}

\end{document}